\documentclass[a4paper,11pt]{amsart}

\usepackage[left=1in,right=1in,top=0.85in,bottom=0.85in]{geometry}
\usepackage{amsmath,amssymb,amsfonts,amsthm}
\usepackage{enumerate}
\usepackage{array}
\usepackage[colorlinks=true,linkcolor=blue,citecolor=blue,urlcolor=blue]{hyperref}

\numberwithin{equation}{section}

\theoremstyle{plain}
\newtheorem{theorem}{Theorem}[section]
\newtheorem{proposition}[theorem]{Proposition}
\newtheorem{corollary}[theorem]{Corollary}
\theoremstyle{definition}
\newtheorem{definition}[theorem]{Definition}
\theoremstyle{remark}
\newtheorem{remark}[theorem]{Remark}

\newcommand{\K}{\mathsf K}
\newcommand{\Cat}{\mathsf{Cat}}
\newcommand{\Set}{\mathsf{Set}}
\newcommand{\GR}{\mathsf R}
\newcommand{\GL}{\mathsf L}
\newcommand{\GJ}{\mathsf J}
\newcommand{\GH}{\mathsf H}
\newcommand{\GD}{\mathsf D}
\newcommand{\Mor}{\mathrm{Mor}}
\newcommand{\Hom}{\mathrm{Hom}}
\newcommand{\id}{\mathrm{id}}
\newcommand{\B}{\mathbf B}
\newcommand{\one}{\mathbf 1}

\title{Comparison categories: the normal form and the regularity of comparison data}
\author{Nizar El Idrissi}
\address{Department of Mathematics, Faculty of Sciences, Ibn Tofail University, Kenitra, Morocco}
\email{nizar.elidrissi@uit.ac.ma}
\subjclass[2020]{18A99, 18N10, 18D12, 20M10, 54E50}
\keywords{equivalence relation, category, 2-category, comparison category, Grothendieck construction, Green's relations, orbit relations, invariants, regularity of comparison data}

\begin{document}

\begin{abstract}
Many equivalence relations of mathematics (Green's relations on a semigroup, orbit relations of group actions, ``same value of an invariant'') are defined by allowed transformations, of a morphism on its two sides or of a point of a set. We put them in one framework. A \emph{comparison datum} produces a category, the \emph{comparison category}, whose arrows are the allowed transformations, and two morphisms are related when they can be compared in both directions. Its \emph{normal form} is a Grothendieck construction, so every such relation is mutual reachability for an action; Green's relations, orbit relations and, more generally, relations defined by an invariant are worked out in detail from it. If the way morphisms are transported is an arbitrary function, every relation occurs, so the real constraint is the \emph{regularity} of the transport. For completely functorial data we show that the relations obtained contain Green's relations and congruences, we study them at two opposite extremes: on finite groups (every morphism invertible) they are described through double cosets, and on nil monoids (only the unit invertible) we prove a collapse theorem. A short plain-language guide to the vocabulary (comparison categories and transport, Grothendieck construction and mutual reachability, functors, pseudofunctors and lax functors, enriched categories, colimits and stabilization) is included for readers who are not specialists.
\end{abstract}

\maketitle

\section{Introduction}

\noindent\textbf{The problem in plain words.} To classify objects one declares some of them to be ``the same''. Very often the rule is given by a transformation. Two matrices are the same if one is obtained from the other by a change of basis. Two elements of a monoid are the same if each divides the other, which is how Green's relations are defined \cite{Green1951,Howie1995}. Two points are the same if a group moves one onto the other, or if they have the same value of an invariant. Each of these relations is usually introduced by its own formula, and the formulas look unrelated. This paper shows that they are one construction.

\smallskip
\noindent\textbf{The idea.} Think of a morphism $m:A\to B$ of a category $\mathcal C$ as the object to be classified. We may modify $m$ on its target side by a morphism $u_1$ and on its source side by a morphism $u_2$, which gives $u_1mu_2$. We say that $m$ \emph{can be compared} to $\tilde m$ if some admissible modification of $m$ \emph{matches} $\tilde m$. In the simplest case matching means equality, $\tilde m=u_1mu_2$. In general the matching takes place in an auxiliary $2$-category $\mathcal D$, into which $m$ and the modifications are transported by three maps $\sigma,\tau_1,\tau_2$, and it is allowed to hold only up to an admissible $2$-morphism. The data $\mathfrak D=(\mathcal C,\mathcal D,\sigma,\tau_1,\tau_2;\mathcal U_1,\mathcal U_2,\mathcal W)$ form a \emph{comparison datum} (Definition \ref{definition-master}). All comparisons are the arrows of a category $\K(\mathfrak D)$, the \emph{comparison category} (Proposition \ref{proposition-K-category}), and the relation we study is
\[
m\simeq\tilde m\iff \text{there is a comparison }m\to\tilde m\text{ and a comparison }\tilde m\to m .
\]
The following dictionary shows how some classical relations appear; they are treated in Sections \ref{section-green} and \ref{section-transport}.
\begin{center}
\footnotesize
\begin{tabular}{>{\raggedright\arraybackslash}p{4.4cm}>{\raggedright\arraybackslash}p{8.2cm}}
\hline
Relation & Which comparisons are allowed \\ \hline
Green's $\GR$ (resp.\ $\GL$) on a monoid $M$ & $\tilde m=mu_2$ (resp.\ $\tilde m=u_1m$), no transport \\
Green's $\GJ$ & $\tilde m=u_1mu_2$, no transport \\
Green's $\GH$ & $\tilde m=u_1m=mu_2$ with \emph{the same} $\tilde m$, via a product trick (Prop.\ \ref{proposition-green-H}) \\
Orbit relation of an action & $\tilde x=x\cdot w$, observations adjoined as morphisms (Cor.\ \ref{corollary-orbits}) \\
Same value of an invariant $\iota$ & $\iota(\tilde x)=\iota(x)$: a transport relation, for the group of permutations preserving $\iota$ (Prop.\ \ref{proposition-invariants}) \\ \hline
\end{tabular}
\end{center}

\smallskip
\noindent\textbf{Main results.}
\begin{itemize}
\item[(A)] \emph{Normal form} (Theorem \ref{theorem-normal-form}). $\K(\mathfrak D)$ is a full subcategory of a Grothendieck construction $\int_VH$, the category of elements of an action. So $m\simeq\tilde m$ is always \emph{mutual reachability for an action, read on the observations $s(m)$}. Moreover, every small category occurs, up to opposite, as a full subcategory of some $\K(\mathfrak D)$ (Corollary \ref{corollary-shape}); the shape of $\K(\mathfrak D)$ is unconstrained.
\item[(B)] \emph{Green's relations}, \emph{orbit relations} and \emph{relations defined by an invariant} are computed from the normal form, with complete proofs (Sections \ref{section-green}, \ref{section-transport}).
\item[(C)] \emph{Regularity.} If $\sigma$ is an arbitrary function, every equivalence relation on a strongly connected category occurs (Proposition \ref{proposition-universality}), so nothing is learned; the constraint lies in how regular $\sigma,\tau_1,\tau_2$ are. For \emph{completely functorial} data (all three are functors) we prove: they contain Green's relations, orbit relations and congruences (Proposition \ref{proposition-functorial-list}); what they produce depends on which morphisms are invertible: on a finite group, where all are, they are exactly the pullbacks along homomorphisms of double coset relations (Proposition \ref{proposition-finite}); on a nil monoid, where only the unit is, they obey an all-or-nothing law at the class of the unit (Theorem \ref{theorem-nil}), which separates them from arbitrary $\sigma$.
\end{itemize}

\smallskip
\noindent\textbf{What is left out.} The relation $\simeq$ is the lowest rung of a ladder whose higher rungs (connected components, homotopy groups of $\K(\mathfrak D)$) are only mentioned in Remark \ref{remark-levels}. Relations defined by a limit or by stabilization are also obtained from comparison data; Section \ref{section-brief} describes them in a few sentences, with examples.

\smallskip
\noindent\textbf{Reader's guide and conventions.} The reader needs only categories and functors; Section \ref{section-guide} explains in plain words the vocabulary that may be unfamiliar, and $2$-categories are recalled in Section \ref{section-master}. Composition is written from right to left, so $u_1mu_2$ means ``first $u_2$, then $m$, then $u_1$''. A monoid $M$ is regarded as the category $\B M$ with one object whose morphisms are the elements of $M$. A subcategory is \emph{wide} if it has all the objects. Sections \ref{section-green} and \ref{section-transport} are independent applications; Section \ref{section-regularity} contains the main new results.

\section{A plain-language guide to the vocabulary}
\label{section-guide}

\noindent This section explains, almost without formulas, five groups of ideas used throughout. Nothing later depends on it logically.

\subsection*{Comparison categories and transport}
To \emph{transport} is to carry something along an arrow: if $w:\tilde A\to A$ and we know something about $A$ (a point, a measurement, a morphism), transport along $w$ gives the corresponding information about $\tilde A$. It may run against the arrow, as when restricting a function to a subspace. Transport is honest when it respects identities and composites; this is what ``functor'' means. Here, the maps $\sigma,\tau_1,\tau_2$ transport a morphism and its modifications to a common place $\mathcal D$ where they can be compared. The \emph{comparison category} $\K(\mathfrak D)$ records the outcome: its objects are the morphisms to be classified, and an arrow is one admissible way of comparing one to another. How faithfully transport respects composition is called \emph{regularity} (Section \ref{section-regularity}).

\subsection*{The Grothendieck construction and mutual reachability}
Suppose a small category $H(v)$ (a \emph{fibre}) sits over each object $v$ of a category $V$, and each arrow of $V$ carries one fibre into another. The Grothendieck construction glues the fibres into one category: an object is a point $v$ with an object of its fibre (a person in a city), and an arrow is a road between cities, followed by an adjustment in the target fibre. If the fibres are sets there is nothing to adjust, and one gets the \emph{category of elements}, whose arrows say ``this element is carried to that one''. For a group acting on a set this is the action groupoid, and two elements are joined exactly when they lie in one orbit. Two objects are \emph{mutually reachable} if there are arrows in both directions. This is the heart of Theorem \ref{theorem-normal-form}: $m\simeq\tilde m$ means that each of $m,\tilde m$ is reachable from the other in such a glued category.

\subsection*{Functors, pseudofunctors, and lax functors}
A functor sends a composite to \emph{exactly} the composite of the images. When the target has $2$-morphisms (arrows between arrows), this is often too strong. A \emph{pseudofunctor} requires only a specified reversible identification between the two routes, as when pulling back along a composite and pulling back in two steps are canonically isomorphic, not equal. A \emph{lax functor} requires only a specified comparison in one direction, like an inequality replacing an equality. The Grothendieck construction extends to both, the identifications being inserted into composition. Here they only name intermediate points of the regularity axis of Section \ref{section-regularity}; nothing specific is proved about them (Open problem (3)).

\subsection*{Enriched categories}
In an \emph{enriched} category the hom-sets are replaced by objects of a richer base, so that the ``set of arrows'' from $A$ to $B$ carries extra structure. The case used here is categories enriched over $[0,\infty]$ \cite{Lawvere1973}: each hom is a number, read as the cost or distance of passing from one object to another, and composition obeys the triangle inequality. A relation then holds \emph{in the limit} if the infimum of the costs of comparisons is $0$, even when no single comparison costs nothing (Section \ref{section-brief}).

\subsection*{Colimits and stabilization}
A \emph{colimit} glues a diagram of objects along the identifications the diagram prescribes. The union of an increasing chain of sets is one, and so is the set of orbits of a group acting on a set. For a family of sets indexed by a category, the colimit is the set of connected components of the Grothendieck construction. A \emph{homotopy colimit} glues in the same way but \emph{remembers} each identification instead of collapsing it; this is why the higher levels of Remark \ref{remark-levels} are homotopy groups of a homotopy colimit. Colimits also express \emph{stabilization}: two objects are \emph{stably} the same if they become the same after adding the same inert piece, as many times as one wishes. The tangent bundle of the $2$-sphere is not trivial, but adding one trivial line makes it trivial. Repeating ``add a spectator'' gives an endless tower of data; its colimit is the world where there is always room to add more, and relations failing at every finite stage can hold there. This is how $K$-theory identifies stably isomorphic modules (Section \ref{section-brief}).

\section{Comparison data and the comparison category}
\label{section-master}

\noindent We first recall what we need about $2$-categories \cite{MacLane1971}. A \emph{$2$-category} $\mathcal D$ has objects, $1$-morphisms $p:X\to Y$ (which compose as in a category) and, for parallel $1$-morphisms $p,q:X\to Y$, \emph{$2$-morphisms} $\alpha:p\Rightarrow q$. Two $2$-morphisms $\alpha:p\Rightarrow q$ and $\beta:q\Rightarrow r$ compose \emph{vertically} to $\beta\circ_v\alpha:p\Rightarrow r$. A $2$-morphism can also be \emph{whiskered}: for $1$-morphisms $a:Y\to Z$ and $b:X'\to X$ there is a $2$-morphism $a\alpha b:apb\Rightarrow aqb$, which is ``$\alpha$ with $a$ composed on the left and $b$ on the right''. Whiskering by a composite is iterated whiskering, $(a'a)\alpha(bb')=a'(a\alpha b)b'$, and it commutes with vertical composition, $a(\beta\circ_v\alpha)b=(a\beta b)\circ_v(a\alpha b)$ (these are consequences of the interchange law). The prototype is the $2$-category of categories, functors and natural transformations. A category is the special case with identity $2$-morphisms only; the reader may keep this case in mind throughout, since the main examples below are of this kind.

\begin{definition}\label{definition-master}
A \textbf{comparison datum} $\mathfrak D=(\mathcal C,\mathcal D,\sigma,\tau_1,\tau_2;\mathcal U_1,\mathcal U_2,\mathcal W)$ consists of a category $\mathcal C$; a $2$-category $\mathcal D$; functors $\tau_1,\tau_2:\mathcal C\to\mathcal D$; a function $\sigma$ with $\sigma(A)=\tau_1(A)=\tau_2(A)$ on objects and $\sigma:\Hom(A,B)\to\Hom(\sigma A,\sigma B)$ on morphisms, which is \textbf{not} required to be a functor; wide subcategories $\mathcal U_1,\mathcal U_2\subseteq\mathcal C$ (the admissible modifications on the target and on the source side); and a class $\mathcal W$ of $2$-morphisms of $\mathcal D$ (the admissible $2$-morphisms) which contains the identities and is closed under vertical composition and whiskering. The datum is \textbf{strict} if $\mathcal W$ consists of the identities, and \textbf{completely functorial} if $\sigma$ is also a functor.
\end{definition}

The roles are as follows. A morphism $m$ of $\mathcal C$ is the object to be classified, and $\sigma(m)$ is its image in $\mathcal D$, where the comparison takes place; $\sigma$ may forget structure, which is why it need not be a functor. The functors $\tau_1,\tau_2$ transport the modifications. For $m:A\to B$, $u_1:B\to\tilde B$ and $u_2:\tilde A\to A$ put
\[
\Theta_{u_1,u_2}(m):=\tau_1(u_1)\,\sigma(m)\,\tau_2(u_2):\ \sigma\tilde A\xrightarrow{\ \tau_2(u_2)\ }\sigma A\xrightarrow{\ \sigma(m)\ }\sigma B\xrightarrow{\ \tau_1(u_1)\ }\sigma\tilde B .
\]
It is parallel to $\sigma(\tilde m)$ whenever $\tilde m:\tilde A\to\tilde B$, so it makes sense to ask for a $2$-morphism $\Theta_{u_1,u_2}(m)\Rightarrow\sigma(\tilde m)$.

\begin{definition}\label{definition-K}
The \textbf{comparison category} $\K(\mathfrak D)$ has the morphisms of $\mathcal C$ as objects. An arrow $m\to\tilde m$ is a triple $(u_1,u_2,\Phi)$ with $u_1\in\mathcal U_1$, $u_2\in\mathcal U_2$, $u_1:B\to\tilde B$, $u_2:\tilde A\to A$, and $\Phi:\Theta_{u_1,u_2}(m)\Rightarrow\sigma(\tilde m)$ a $2$-morphism in $\mathcal W$. Composition and identities are
\[
(u_1',u_2',\Phi')\circ(u_1,u_2,\Phi)=\bigl(u_1'u_1,\;u_2u_2',\;\Phi'\circ_v(\tau_1(u_1')\,\Phi\,\tau_2(u_2'))\bigr),\qquad (\id,\id,1_{\sigma(m)}).
\]
\end{definition}

\begin{proposition}\label{proposition-K-category}
$\K(\mathfrak D)$ is a category.
\end{proposition}
\begin{proof}
\emph{The composite is well defined.} Let $m:A\to B$, $\tilde m:\tilde A\to\tilde B$, $\hat m:\hat A\to\hat B$ and let $(u_1,u_2,\Phi):m\to\tilde m$, $(u_1',u_2',\Phi'):\tilde m\to\hat m$ be arrows. Then $u_1'u_1:B\to\hat B$, and $u_2u_2':\hat A\to A$ (the order is reversed on the source side). Whiskering $\Phi:\Theta_{u_1,u_2}(m)\Rightarrow\sigma(\tilde m)$ by $\tau_1(u_1')$ on the left and $\tau_2(u_2')$ on the right gives
\[
\tau_1(u_1')\Theta_{u_1,u_2}(m)\tau_2(u_2')\Longrightarrow\tau_1(u_1')\sigma(\tilde m)\tau_2(u_2')=\Theta_{u_1',u_2'}(\tilde m).
\]
Since $\tau_1,\tau_2$ are functors, the source is $\tau_1(u_1'u_1)\sigma(m)\tau_2(u_2u_2')=\Theta_{u_1'u_1,u_2u_2'}(m)$, and the target is the source of $\Phi'$. So $\Phi'\circ_v(\tau_1(u_1')\Phi\tau_2(u_2'))$ is defined, is a $2$-morphism $\Theta_{u_1'u_1,u_2u_2'}(m)\Rightarrow\sigma(\hat m)$, and lies in $\mathcal W$ because $\mathcal W$ is closed under whiskering and vertical composition. Also $u_1'u_1\in\mathcal U_1$ and $u_2u_2'\in\mathcal U_2$ because these are subcategories. The identity $(\id,\id,1_{\sigma(m)})$ is an arrow because the $\tau_i$ preserve identities, so $\Theta_{\id,\id}(m)=\sigma(m)$, and $1_{\sigma(m)}\in\mathcal W$.

\emph{Associativity and units.} For a third arrow $(u_1'',u_2'',\Phi'')$, both bracketings of the triple composite have first two components $u_1''u_1'u_1$ and $u_2u_2'u_2''$, and third component $\Phi''\circ_v(\tau_1(u_1'')\Phi'\tau_2(u_2''))\circ_v(\tau_1(u_1''u_1')\Phi\tau_2(u_2'u_2''))$. Indeed $\circ_v$ is associative; whiskering by a composite is iterated whiskering; and whiskering commutes with $\circ_v$. For the same reasons $(\id,\id,1_{\sigma(m)})$ is neutral.
\end{proof}

\begin{definition}\label{definition-relation}
Write $m\trianglelefteq\tilde m$ if $\Hom_{\K(\mathfrak D)}(m,\tilde m)\neq\emptyset$, and $m\simeq\tilde m$ if $m\trianglelefteq\tilde m$ and $\tilde m\trianglelefteq m$. Explicitly, $m\simeq\tilde m$ iff there are admissible $u_1,u_2,v_1,v_2$ and admissible $2$-morphisms $\Theta_{u_1,u_2}(m)\Rightarrow\sigma(\tilde m)$ and $\Theta_{v_1,v_2}(\tilde m)\Rightarrow\sigma(m)$.
\end{definition}

\begin{proposition}\label{proposition-preorder}
$\trianglelefteq$ is a preorder and $\simeq$ is an equivalence relation on $\Mor(\mathcal C)$. If $\K(\mathfrak D)$ is a groupoid, $\trianglelefteq$ is itself symmetric and equals $\simeq$.
\end{proposition}
\begin{proof}
Reflexivity comes from the identity arrows and transitivity from composition (Proposition \ref{proposition-K-category}); symmetry of $\simeq$ is built into its definition. In a groupoid every arrow has an inverse arrow.
\end{proof}

\begin{remark}\label{remark-levels}
The relation $\simeq$ only asks whether comparisons exist, in both directions; it is the lowest (level $-1$) of a sequence of ``levels'' read from $\K(\mathfrak D)$. Level $0$ is the relation generated by comparisons in either direction, that is, the connected components of $\K(\mathfrak D)$; it is coarser than $\simeq$, and equal to it when $\K(\mathfrak D)$ is a groupoid. Levels 1 and higher are the homotopy groups $\pi_n$ of the classifying space of $\K(\mathfrak D)$ at $m$. For instance, level $1$ is the group of coherent self-comparisons of $m$ (for a group acting on a set, the stabilizer). These levels count \emph{in how many ways} two morphisms are comparable; for transport data they are the homotopy groups of a homotopy colimit \cite{Thomason1979}. We do not use them below.
\end{remark}

\section{The normal form}
\label{section-normal}

\noindent A comparison looks complicated, since it combines two modifications and a $2$-morphism. It is in fact a very classical object: an arrow of the category of elements of an action (see Section \ref{section-guide} for an informal description of the Grothendieck construction). We first build the action.

\begin{definition}\label{definition-normal}
Let $\mathfrak D$ be a comparison datum. Let $V$ be the category whose objects are the pairs $(A,B)$ of objects of $\mathcal C$, and whose arrows $(A,B)\to(\tilde A,\tilde B)$ are the pairs $(u_2,u_1)$ with $u_2:\tilde A\to A$ in $\mathcal U_2$ and $u_1:B\to\tilde B$ in $\mathcal U_1$, composed by $(u_2',u_1')(u_2,u_1)=(u_2u_2',u_1'u_1)$. Let $H:V\to\Cat$ send $(A,B)$ to the category $H(A,B)$ whose objects are the $1$-morphisms $\sigma A\to\sigma B$ of $\mathcal D$ and whose arrows are the $2$-morphisms of $\mathcal W$; it sends $(u_2,u_1)$ to the functor $p\mapsto\tau_1(u_1)\,p\,\tau_2(u_2)$, acting on $2$-morphisms by whiskering. The \textbf{Grothendieck construction} $\int_VH$ has as objects the pairs $(v,x)$ with $x\in H(v)$, and as arrows $(v,x)\to(v',x')$ the pairs $(a,\varphi)$ with $a:v\to v'$ in $V$ and $\varphi:H(a)x\to x'$ in $H(v')$, composed by $(a',\varphi')(a,\varphi)=(a'a,\ \varphi'\circ_vH(a')\varphi)$. For a morphism $m:A\to B$ of $\mathcal C$ put $s(m):=\bigl((A,B),\sigma(m)\bigr)$.
\end{definition}

\noindent In words: $V$ is a category that acts on the $1$-morphisms of $\mathcal D$ by modifying them on both sides, $H$ is this action, and $\int_VH$ is the category whose arrows say ``$x$ is moved by $a$ to something that is related by $\varphi$ to $x'$''. The functor property of $H$ is the functoriality of $\tau_1,\tau_2$, and $H(a)$ preserves the arrows of $\mathcal W$ because $\mathcal W$ is closed under whiskering. The point is that $V$ and $H$ depend on $\sigma$ only through the objects $\sigma A$; $\sigma$ itself enters only through the map $s$.

\noindent The comparison category is a piece of this construction:

\begin{theorem}\label{theorem-normal-form}
For every comparison datum $\mathfrak D$ and all morphisms $m,\tilde m$ of $\mathcal C$,
\[
\Hom_{\K(\mathfrak D)}(m,\tilde m)=\Hom_{\int_VH}\bigl(s(m),s(\tilde m)\bigr),
\]
compatibly with composition and identities. Hence $m\mapsto s(m)$ is a fully faithful functor $\K(\mathfrak D)\to\int_VH$, an equivalence onto the full subcategory on the objects $s(m)$, and
\[
m\simeq\tilde m\iff s(m)\text{ and }s(\tilde m)\text{ are mutually reachable in }\textstyle\int_VH .
\]
\end{theorem}
\begin{proof}
An arrow $s(m)\to s(\tilde m)$ in $\int_VH$ is a pair $\bigl((u_2,u_1),\Phi\bigr)$ with $u_2:\tilde A\to A$ in $\mathcal U_2$, $u_1:B\to\tilde B$ in $\mathcal U_1$ and $\Phi:\tau_1(u_1)\sigma(m)\tau_2(u_2)\Rightarrow\sigma(\tilde m)$ in $\mathcal W$, that is, $\Phi:\Theta_{u_1,u_2}(m)\Rightarrow\sigma(\tilde m)$: this is exactly a triple $(u_1,u_2,\Phi)$ of Definition \ref{definition-K}. The composition of $\int_VH$ gives $(u_2u_2',u_1'u_1)$ and $\Phi'\circ_vH(u_2',u_1')\Phi$, and $H(u_2',u_1')\Phi=\tau_1(u_1')\Phi\tau_2(u_2')$, which is the composition of Definition \ref{definition-K}; the identities agree. If $s(m)=s(m')$ then $m,m'$ are parallel with $\sigma m=\sigma m'$, so $(\id,\id,1)$ is an isomorphism $m\to m'$; hence $\K(\mathfrak D)$ is equivalent to its image. The last statement is the definition of $\simeq$.
\end{proof}

\noindent For a strict datum the category $H(A,B)$ is a set, $V$ acts on the set $Y$ of all $1$-morphisms $\sigma A\to\sigma B$, and the theorem says: \emph{$m\simeq\tilde m$ iff $s(\tilde m)$ is in the $V$-orbit of $s(m)$ and $s(m)$ is in the $V$-orbit of $s(\tilde m)$}. So every relation in this paper is the pullback along $s$ of an orbit relation. The freedom of a datum lies in the map $s$, hence in $\sigma$. The next section shows how the category $V$ and the set $Y$ arise in concrete cases; first we record that the shape of $\K(\mathfrak D)$ is unconstrained.

\begin{corollary}\label{corollary-shape}
Every small category $\mathcal U$ is, up to opposite, a full subcategory of $\K(\mathfrak D)$ for a completely functorial strict datum $\mathfrak D$, and this subcategory is a union of connected components.
\end{corollary}
\begin{proof}
Apply Theorem \ref{theorem-transport} below to $\mathcal C=\mathcal U$ and the constant functor $F=*$: the observations form a full subcategory isomorphic to $\mathcal U^{\mathrm{op}}$, and no arrow joins an observation to another object.
\end{proof}

\noindent Consequently no restriction on $\K(\mathfrak D)$ can explain which relations occur, and the constraint must come from $s$, that is, from how regular $\sigma,\tau_1,\tau_2$ are. This is Section \ref{section-regularity}.

\section{First construction: Green's relations}
\label{section-green}

\noindent Green's relations are the oldest relations of the kind studied here. Let $M$ be a monoid. For $m\in M$ let $mM$, $Mm$, $MmM$ be the principal right, left and two-sided ideals generated by $m$. Green's relations \cite{Green1951,Howie1995} are
\[
m\,\GR\,\tilde m\iff mM=\tilde mM,\qquad m\,\GL\,\tilde m\iff Mm=M\tilde m,\qquad m\,\GJ\,\tilde m\iff MmM=M\tilde mM,
\]
and $\GH=\GR\cap\GL$. Concretely, $m\,\GR\,\tilde m$ means $\tilde m=mx$ and $m=\tilde my$ for some $x,y$: each is a right multiple of the other; $\GJ$ allows two-sided multiples. They are the basic tools of semigroup theory: the $\GH$-classes containing an idempotent are the maximal subgroups, and $\GJ$ orders elements by ``size''.

\begin{definition}\label{definition-green-data}
For a category $\mathcal C$ regarded as a $2$-category with identity $2$-cells only, the \textbf{Green datum} is $\mathrm{Gr}(\mathcal C)=(\mathcal C,\mathcal C,\id,\id,\id;\ \mathcal U_1=\mathcal U_2=\mathcal C,\ \mathcal W=\one)$. Let $\mathrm{Gr}_{\GR}(\mathcal C)$ and $\mathrm{Gr}_{\GL}(\mathcal C)$ be obtained by replacing $\mathcal U_1$, respectively $\mathcal U_2$, by the subcategory $\one$ of identities.
\end{definition}

\noindent Nothing is transported ($\sigma=\tau_i=\id$), so these data are strict and completely functorial.

\begin{proposition}\label{proposition-green}
\begin{enumerate}[(a)]
\item In $\K(\mathrm{Gr}(\mathcal C))$ an arrow $m\to\tilde m$ is a pair $(u_1,u_2)$ with $\tilde m=u_1mu_2$. Hence $m\simeq\tilde m$ iff $\tilde m=u_1mu_2$ and $m=v_1\tilde mv_2$ for some morphisms $u_i,v_i$.
\item For $\mathcal C=\B M$, the relation $\simeq$ of $\mathrm{Gr}(\mathcal C)$ is $\GJ$, that of $\mathrm{Gr}_{\GR}$ is $\GR$ and that of $\mathrm{Gr}_{\GL}$ is $\GL$.
\end{enumerate}
\end{proposition}
\begin{proof}
(a) Here $\Theta_{u_1,u_2}(m)=u_1mu_2$, and a $2$-morphism $u_1mu_2\Rightarrow\tilde m$ in $\one$ exists iff $u_1mu_2=\tilde m$. (b) For $\mathcal C=\B M$, $m\trianglelefteq\tilde m$ means $\tilde m\in MmM$. So $m\simeq\tilde m$ iff $\tilde m\in MmM$ and $m\in M\tilde mM$; since an ideal containing the generator of another ideal contains that ideal, this says $MmM=M\tilde mM$. For $\mathrm{Gr}_{\GR}$ we have $u_1=1$, so $m\simeq\tilde m$ iff $\tilde m\in mM$ and $m\in\tilde mM$, that is, $mM=\tilde mM$; $\mathrm{Gr}_{\GL}$ is symmetric.
\end{proof}

\noindent By Theorem \ref{theorem-normal-form}, this says that \emph{$\GJ$ is the orbit relation of the two-sided action of $\mathcal C$ on its own morphisms}: here $V=\mathcal C^{\mathrm{op}}\times\mathcal C$ acts on $Y=\Mor(\mathcal C)$ by $(u_2,u_1)\cdot p=u_1pu_2$, and $s$ is a bijection onto $Y$.

\smallskip
\noindent\textbf{The relation $\GH$ by a product trick.} Since $\GH=\GR\cap\GL$ requires the \emph{same} element $\tilde m$ to be reached by a left and by a right multiplication, one cannot simply combine $\mathrm{Gr}_{\GR}$ and $\mathrm{Gr}_{\GL}$. The auxiliary $2$-category solves this: it carries two coordinates, and each modification acts on one coordinate only.

\begin{proposition}\label{proposition-green-H}
Let $M$ be a monoid, $\mathrm{Gr}_{\GH}(M)=(\B M,\B(M\times M),\sigma,\tau_1,\tau_2;\ \B M,\B M,\one)$ with the monoid homomorphisms $\sigma(m)=(m,m)$, $\tau_1(u)=(u,1)$, $\tau_2(u)=(1,u)$. Then the relation $\simeq$ of $\mathrm{Gr}_{\GH}(M)$ is Green's relation $\GH$.
\end{proposition}
\begin{proof}
An arrow $m\to\tilde m$ is a pair $(u_1,u_2)$ with $(u_1,1)(m,m)(1,u_2)=(\tilde m,\tilde m)$, that is, $\tilde m=u_1m=mu_2$. So $m\trianglelefteq\tilde m$ iff $\tilde m\in Mm\cap mM$, and $m\simeq\tilde m$ iff $\tilde m\in Mm\cap mM$ and $m\in M\tilde m\cap\tilde mM$. The conditions $\tilde m\in Mm$, $m\in M\tilde m$ say $Mm=M\tilde m$, and the other two say $mM=\tilde mM$.
\end{proof}

\noindent The remaining relation $\GD=\GR\vee\GL$ asks for an intermediate element $n$ with $m\,\GR\,n\,\GL\,\tilde m$; the two comparisons do not share a target and the trick does not apply. Whether $\GD$ is of the form $\simeq$ for some comparison datum is open. Finally, $\GJ$ is not compatible with the product of $M$ in general, so it is not a congruence (Section \ref{section-brief}); the relations that do arise from functorial data and are congruences are examined in Section \ref{section-regularity}.

\section{Second construction: transport relations, orbits and invariants}
\label{section-transport}

\noindent The relations obtained from a group or monoid acting on a set (``two points are the same if one is moved to the other'') are not Green data on the object to be classified, because the points of the set are not morphisms. The trick is to make them morphisms by adjoining one object. It also covers relations defined by an invariant (Proposition \ref{proposition-invariants}).

\begin{definition}\label{definition-CF}
Let $F:\mathcal C^{\mathrm{op}}\to\Set$ be a functor: it assigns to each object $A$ a set $F(A)$ of \emph{observations on $A$}, and to each $w:\tilde A\to A$ a map $w^*:F(A)\to F(\tilde A)$ with $\id^*=\id$ and $(ww')^*=w'^*w^*$. Let $\mathcal U\subseteq\mathcal C$ be a wide subcategory. The category $\mathcal C_F$ has the objects of $\mathcal C$ and one new object $\star$, and
\[
\Hom(A,B)=\Hom_{\mathcal C}(A,B),\quad \Hom(A,\star)=F(A),\quad \Hom(\star,\star)=\{\id\},\quad \Hom(\star,A)=\emptyset,
\]
with composition $x\circ w:=w^*x$ for an observation $x:A\to\star$ and $w:\tilde A\to A$. Associativity $(x\circ w)\circ w'=x\circ(ww')$ is the identity $w'^*w^*=(ww')^*$. The \textbf{transport datum} is
\[
\mathfrak D_F^{\mathcal U}=\bigl(\mathcal C_F,\mathcal C_F,\id,\id,\id;\ \mathcal U_1=\one,\ \mathcal U_2=\mathcal U\cup\{\id_\star\},\ \mathcal W=\one\bigr).
\]
\end{definition}

\noindent It is a Green-type datum on the augmented category $\mathcal C_F$, hence strict and completely functorial. Among its objects, the morphisms of $\mathcal C_F$, the \emph{observations} $x:A\to\star$ are the ones of interest. Since $\mathcal U_1=\one$, nothing can be done on the target side, which forces the target $\star$ to stay put.

\begin{theorem}\label{theorem-transport}
For observations $x\in F(A)$ and $\tilde x\in F(\tilde A)$, an arrow $x\to\tilde x$ in $\K(\mathfrak D_F^{\mathcal U})$ is a morphism $w:\tilde A\to A$ of $\mathcal U$ with $w^*x=\tilde x$, and composition is composition of such morphisms. So the full subcategory on the observations is the category of elements $\int_{\mathcal U}F$, and
\[
x\simeq\tilde x\iff\text{there are }w,w'\in\mathcal U\text{ with }w^*x=\tilde x\text{ and }w'^*\tilde x=x .
\]
Moreover no arrow of $\K(\mathfrak D_F^{\mathcal U})$ joins an observation to an object that is not an observation, so the observations form a union of connected components.
\end{theorem}
\begin{proof}
By Proposition \ref{proposition-green}(a) an arrow $x\to\tilde x$ is a pair $(u_1,u_2)$ with $\tilde x=u_1\circ x\circ u_2$, where $u_1\in\one$ and $u_2\in\mathcal U\cup\{\id_\star\}$. Both $x$ and $\tilde x$ have target $\star$, so $u_1=\id_\star$. A morphism $u_2:\tilde A\to A$ with $A,\tilde A\neq\star$ must lie in $\mathcal U$, and then $\tilde x=x\circ u_2=u_2^*x$. Composition of arrows is $(w,w')\mapsto ww'$, as $(u_1',u_2')\circ(u_1,u_2)=(u_1'u_1,u_2u_2')$. Finally, an arrow between $\id_\star$ and an observation would need a morphism of $\mathcal U_2$ between $\star$ and an object $A\ne\star$, and there is none; and $u_1=\id$ is forced.
\end{proof}

\noindent Orbit relations are the special case of one-object categories:

\begin{corollary}\label{corollary-orbits}
Let a monoid $M$ act on the right on a set $Z$, and take $\mathcal C=\mathcal U=\B M$, $F(*)=Z$, $F(m)=(z\mapsto zm)$. Then $z\simeq z'$ iff $zM=z'M$. If $G$ is a group acting on the left, with $F(g)z=g^{-1}z$, then $\simeq$ is the orbit relation $z'\in G\cdot z$, and $\K$ is the action groupoid.
\end{corollary}
\begin{proof}
By Theorem \ref{theorem-transport}, $z\simeq z'$ iff $z'=zm$ and $z=z'm'$ for some $m,m'\in M$. Then $z'M=zmM\subseteq zM$ and symmetrically $zM\subseteq z'M$, so $zM=z'M$; the converse is clear. For a group, $F(g)z=g^{-1}z$ is the contravariant functor required, every arrow $g$ has the inverse $g^{-1}$, and $z\simeq z'$ iff $z'=g^{-1}z$ for some $g$.
\end{proof}

\noindent Relations defined by an invariant are transport relations too:

\begin{proposition}\label{proposition-invariants}
Let $\iota:Y\to I$ be any map (an \emph{invariant}) and $G$ the group of permutations $g$ of $Y$ with $\iota g=\iota$. For $\mathcal C=\mathcal U=\B G$ and $F(*)=Y$ with $F(g)y=g^{-1}y$, the transport relation is $y\simeq y'\iff\iota(y)=\iota(y')$.
\end{proposition}
\begin{proof}
$G$ preserves the fibres of $\iota$, and the transposition of two points of one fibre lies in $G$; so the orbits are the fibres. Apply Corollary \ref{corollary-orbits}.
\end{proof}

\noindent \textbf{Why this matters.} First, relations defined by an invariant (such as rank) and orbit relations (``same orbit'') are \emph{the same class}: transport relations. Second, $F=*$ gives $\int_{\mathcal U}F=\mathcal U^{\mathrm{op}}$, which proves Corollary \ref{corollary-shape}. Third, every equivalence relation on a set is a transport relation, restricted to the observations (take $\iota$ the quotient map); this is why ``every relation occurs'' is an empty statement, and it leads to the next section.

\section{The regularity axis, and completely functorial data}
\label{section-regularity}

\noindent \textbf{Why a constraint is needed.} If $\sigma$ may be an arbitrary function it can simply record the answer. Call $\mathcal C$ \emph{strongly connected} if for any two objects $A,B$ there is a morphism $A\to B$ (for instance a monoid, or the category of vector spaces).

\begin{proposition}\label{proposition-universality}
Let $\mathcal C$ be strongly connected and $E$ any equivalence relation on $\Mor(\mathcal C)$. There is a comparison datum with $(\mathcal U_1,\mathcal U_2,\mathcal W)=(\mathcal C,\mathcal C,\one)$ for which $\simeq$ is $E$.
\end{proposition}
\begin{proof}
Let $M=\{1,0\}\sqcup\Mor(\mathcal C)/E$ with $1$ a unit and $xy=0$ whenever $x,y\ne1$; it is a monoid (if no factor of a triple product is $1$ both bracketings are $0$, and factors equal to $1$ can be removed). Let $\mathcal D$ have the objects of $\mathcal C$ and $\Hom_{\mathcal D}(A,B)=M$ for all $A,B$, composed by the product, with identity $2$-cells only. Let $\tau_1=\tau_2$ send every morphism to $1$, and $\sigma(m)=[m]$. Then $\Theta_{u_1,u_2}(m)=[m]$, so an arrow $m\to\tilde m$ is a pair of morphisms $B\to\tilde B$, $\tilde A\to A$ with $[m]=[\tilde m]$. These exist because $\mathcal C$ is strongly connected, so $m\trianglelefteq\tilde m$ iff $m\mathrel{E}\tilde m$.
\end{proof}

\noindent This normalizes the framework; it does not classify. The meaningful question is: \emph{how regular must $\sigma,\tau_1,\tau_2$ be, and which relations survive at each degree of regularity?} We grade data along the following axis, from most to least regular.
\begin{itemize}
\item \emph{Rigid}: $\mathcal D=\mathcal C$, $\sigma=\tau_1=\tau_2=\id$, identity $2$-cells (nothing is transported; the Green data of Section \ref{section-green}, with $\mathcal U_1,\mathcal U_2$ arbitrary wide subcategories).
\item \emph{Diagonal}: $\sigma=\tau_1=\tau_2$ is a single functor, identity $2$-cells.
\item \emph{Completely functorial}: $\sigma,\tau_1,\tau_2$ are three (possibly different) functors.
\item \emph{Pseudofunctorial and lax}: $\sigma$ preserves composition only up to an admissible $2$-isomorphism, respectively $2$-morphism (Section \ref{section-guide}).
\item \emph{Arbitrary}: $\sigma$ is any function (Proposition \ref{proposition-universality}).
\end{itemize}
The two most regular classes can be classified completely, by the same argument as Proposition \ref{proposition-green} and Theorem \ref{theorem-normal-form}: a rigid datum gives the relation ``$\tilde m=u_1mu_2$ and $m=v_1\tilde mv_2$ with $u_i,v_i$ in the given subcategories'', and a diagonal datum gives the same relation read in the quotient $\mathcal C/{\approx}$ by the congruence $\approx=\ker\sigma$ and pulled back. We now concentrate on the completely functorial class, where genuinely new relations first appear. \emph{From now on all data are strict and completely functorial}, and we write $\mathcal F(\mathcal C)$ for the set of relations $\simeq$ they produce on $\Mor(\mathcal C)$.

\smallskip
\looseness=-1 \noindent\textbf{How the subsections fit together.} Subsection \ref{subsection-lower} gives lower bounds valid for every $\mathcal C$. The next two take one-object categories $\B M$, where everything is governed by \emph{which elements of $M$ are invertible}, since a functor preserves invertible morphisms and nothing more. In Subsection \ref{subsection-finite} all are invertible (finite groups): comparisons are reversible and $\mathcal F(\B G)$ is described exactly. In Subsection \ref{subsection-nil} only the unit is (nil monoids), and $\mathcal F(\B M)$ misses relations that arbitrary $\sigma$ realizes.

\subsection{Lower bounds: what completely functorial data contain}\label{subsection-lower}

\noindent By Theorem \ref{theorem-normal-form}, for a completely functorial datum the map $s$ is compatible with composition in $\mathcal C$: $s(amb)=\bigl((A',B'),\sigma(a)\sigma(m)\sigma(b)\bigr)$. This is what arbitrary $\sigma$ lacks. Known examples:

\begin{proposition}\label{proposition-functorial-list}
$\mathcal F(\mathcal C)$ contains: (a) $\GR,\GL,\GJ$ and, for monoids, $\GH$; (b) the transport relations of Theorem \ref{theorem-transport} (hence all orbit relations); (c) every congruence of $\mathcal C$, that is, every equivalence relation on $\Mor(\mathcal C)$ on parallel morphisms compatible with composition on both sides.
\end{proposition}
\begin{proof}
(a), (b) are Propositions \ref{proposition-green}, \ref{proposition-green-H} and Theorem \ref{theorem-transport}. (c) Let $E$ be a congruence and $Q:\mathcal C\to\mathcal C/E$ the quotient functor. Take $\sigma=\tau_i=Q$, $\mathcal U_1=\mathcal U_2=\one$, $\mathcal W=\one$. An arrow $m\to\tilde m$ forces $u_i=\id$ and $Qm=Q\tilde m$, that is, $m\mathrel{E}\tilde m$.
\end{proof}

\looseness=-1 \noindent These relations come from rigid or diagonal data or from transport. A further source is a datum in which $\mathcal D$ remembers a \emph{constraint} tying the two modifications together. For monoids $M,\Gamma$ and homomorphisms $\beta_1,\beta_2:M\to\Gamma$, take $\mathcal D=\B(M\times\Gamma)$, $\sigma(m)=(m,1)$, $\tau_i(u)=(u,\beta_i(u))$ and $\mathcal U_1=\mathcal U_2=\B M$. Then $\tau_1(u_1)\sigma(m)\tau_2(u_2)=(u_1mu_2,\beta_1(u_1)\beta_2(u_2))$, so an arrow $m\to\tilde m$ is a pair $(u_1,u_2)$ with
\[
\tilde m=u_1mu_2\qquad\text{and}\qquad\beta_1(u_1)\beta_2(u_2)=1 .
\]
The constraint is preserved by composition, and for a group $M$ these pairs form a subgroup $P\subseteq M\times M^{\mathrm{op}}$ (if $\beta_1(u_1)\beta_2(u_2)=1$ then also $\beta_2(u_2)\beta_1(u_1)=1$, so $(u_1^{-1},u_2^{-1})\in P$). Hence $\simeq$ is then the orbit relation of $P$ acting by $(u_1,u_2)\cdot m=u_1mu_2$, an orbit relation in the sense of Section \ref{section-transport}; what is new is that the two sides are tied, which no rigid datum can express.

\subsection{The case of finite groups (every morphism is invertible)}\label{subsection-finite}

\looseness=-1 \noindent On a one-object category $\B M$ ($M$ a monoid) a completely functorial strict datum is simple: $\sigma,\tau_1,\tau_2$ are monoid homomorphisms $M\to N$, where $N$ is the monoid of $1$-morphisms $\sigma(*)\to\sigma(*)$ of $\mathcal D$, and $\mathcal U_1,\mathcal U_2$ are submonoids of $M$. Since $\mathcal W$ consists of identities, Definition \ref{definition-K} gives, for every monoid $M$,
\begin{equation}\label{eq-criterion}
m\trianglelefteq\tilde m\iff\sigma(\tilde m)=\tau_1(u_1)\,\sigma(m)\,\tau_2(u_2)\ \text{ for some }u_1\in\mathcal U_1,\ u_2\in\mathcal U_2 .
\end{equation}
In a group every comparison can be undone, so $\simeq$ equals $\trianglelefteq$ (Proposition \ref{proposition-preorder}); with few units this fails (Subsection \ref{subsection-nil}). We start with the favourable case; finiteness only ensures that submonoids of $G$ are subgroups.

\begin{proposition}\label{proposition-finite}
Let $G$ be a finite group and $\mathfrak D$ a completely functorial strict datum on $\B G$, with $\sigma,\tau_1,\tau_2:G\to N$ as above. Then $\K(\mathfrak D)$ is a groupoid, and
\[
m\simeq\tilde m\iff\sigma(\tilde m)\in A\,\sigma(m)\,B,\qquad A=\tau_1(\mathcal U_1),\ B=\tau_2(\mathcal U_2),
\]
where $A,B$ are subgroups of the group $N^\times$ of invertible elements of $N$. That is, $\simeq$ is the pullback along $\sigma$ of a double coset relation.
\end{proposition}
\begin{proof}
A submonoid of a finite group is a subgroup, so $\mathcal U_1,\mathcal U_2$ are subgroups of $G$, and $A,B$ are subgroups of $N^\times$ because homomorphisms send $G$ into the units. By \eqref{eq-criterion}, $m\trianglelefteq\tilde m$ iff $\sigma(\tilde m)\in A\sigma(m)B$. This is symmetric: $\sigma(\tilde m)=a\sigma(m)b$ gives $\sigma(m)=a^{-1}\sigma(\tilde m)b^{-1}$, with the arrow $(u_1^{-1},u_2^{-1})$. So every arrow of $\K(\mathfrak D)$ is invertible.
\end{proof}

\looseness=-1 \noindent These relations exceed those of rigid or diagonal data, because $\sigma,\tau_1,\tau_2$ may be three \emph{different} homomorphisms and the double cosets live in $N$, not in $G$. For instance, on $\mathbb Z/3$ the relation with classes $\{0,1\},\{2\}$ is obtained with three homomorphisms into $A_4$ (take $\sigma(1)$ a $3$-cycle $c$ and $\tau_1,\tau_2$ onto two different subgroups $A,B$ of order $3$ with $c\in AB$, $c^{-1}\notin AB$; checked by computer). It is neither a congruence nor a double coset relation of $\mathbb Z/3$, since on an abelian group those have classes of equal size; so $\mathcal F(\mathcal C)$ is strictly larger than the rigid and diagonal classes. All five equivalence relations on $\mathbb Z/3$, and all $15$ on $\mathbb Z/4$ (by exhaustive search), are obtained, and we know no relation on a finite group which is not (Open problem (1)).

\subsection{The case of nil monoids (only the unit is invertible)}\label{subsection-nil}

\looseness=-1 \noindent In Subsection \ref{subsection-finite} everything was reversible because every element of $G$ is invertible. We now go to the opposite extreme, still using \eqref{eq-criterion}: monoids in which the unit is the only invertible element. There, completely functorial data cannot realize every relation. Call a monoid $M$ a \emph{nil monoid} if it is nontrivial, has a zero $0$, and every element other than $1$ is nilpotent. The basic example is $\{1,a,0\}$ with $a^2=0$. In a nil monoid the only unit is $1$ (an invertible nilpotent element would force $0=1$, hence $M$ trivial), and $0\neq1$. On nil monoids, completely functorial data collapse as follows.

\begin{theorem}\label{theorem-nil}
Let $M$ be a nil monoid and $E\in\mathcal F(\B M)$. If the class of $1$ is not $\{1\}$, then $E$ is universal. In particular, on $M=\{1,a,0\}$ the relation $E_0$ with classes $\{1,a\},\{0\}$ is not in $\mathcal F(\B M)$.
\end{theorem}
\begin{proof}
Let $\mathfrak D$ be completely functorial and strict with $1\simeq a$, $a\neq1$. The maps $\sigma,\tau_i$ are homomorphisms $M\to N$ (Subsection \ref{subsection-finite}). By \eqref{eq-criterion}, an arrow $1\to a$ gives $u_1,u_2$ with $\tau_1(u_1)\tau_2(u_2)=\sigma(a)$, and an arrow $a\to1$ gives $v_1,v_2$ with $\tau_1(v_1)\sigma(a)\tau_2(v_2)=1$. Hence
\[
\tau_1(p)\tau_2(q)=1,\qquad p=v_1u_1,\quad q=u_2v_2 .
\]

\emph{(A) If $p\ne1$ then $\tau_1\equiv1$.} Then $p^n=0$ for some $n$. As $\tau_1(p)$ has the right inverse $\tau_2(q)$, so has $e=\tau_1(p)^n=\tau_1(0)$, which is idempotent. If $ey=1$ then $e=e(ey)=ey=1$. So $\tau_1(0)=1$, and $\tau_1(x)=\tau_1(x)\tau_1(0)=\tau_1(x0)=\tau_1(0)=1$ for all $x$.

\emph{(B) If $q\neq1$ then $\tau_2\equiv1$.} Symmetric: $\tau_2(q)$ has the left inverse $\tau_1(p)$, so does the idempotent $f=\tau_2(0)$, say $yf=1$, whence $f=(yf)f=yf=1$ and $\tau_2\equiv1$ as in (A).

\emph{(C) If $p=1$ then $u_1=v_1=1$; if $q=1$ then $u_2=v_2=1$.} Suppose $v_1u_1=1$ and put $e=u_1v_1$, an idempotent. If $e=1$ then $u_1$ is a unit, so $u_1=v_1=1$. Otherwise $e$ is nilpotent, hence $e=e^n=0$, and $u_1=u_1v_1u_1=eu_1=0$, so $1=v_1u_1=0$, impossible. The second assertion is symmetric.

By (A) and (C), $\tau_1(u_1)=1$ in both cases, and by (B) and (C), $\tau_2(u_2)=1$. So $\sigma(a)=\tau_1(u_1)\tau_2(u_2)=1$. Since $a^n=0$ for some $n$, $\sigma(0)=1$, and $\sigma(x)=\sigma(x)\sigma(0)=\sigma(x0)=\sigma(0)=1$ for every $x$. So $\sigma\equiv1$, and $(\id,\id)$ is an arrow $m\to\tilde m$ for all $m,\tilde m$: $E$ is universal.
\end{proof}

\noindent The mechanism is that homomorphisms cannot make a non-unit behave like a unit without collapsing its powers, and in a nil monoid the powers reach the absorbing $0$: $E_0$ identifies $a$ with $1$ but not $a^2=0$ with $1^2=1$. Since $\B M$ is strongly connected, $E_0$ is realized by an arbitrary $\sigma$ (Proposition \ref{proposition-universality}); so completely functorial data are strictly weaker than arbitrary ones. Together with Subsection \ref{subsection-finite}, this shows that $\mathcal F(\mathcal C)$ lies strictly between the Green-type and the arbitrary classes, the supply of invertible morphisms deciding how much of it is available. The obstruction is specific to the relation $\simeq$, which asks for comparisons in both directions: for the coarser relation generated by comparisons in either direction (level $0$ of Remark \ref{remark-levels}), $E_0$ is realized by an explicit completely functorial datum into a non-commutative monoid presented by generators $X,Y$ and relations $X^3=X^2$, $Y^3=Y^2$, $(XY)^3=(XY)^2$.

\section{Other constructions and examples, in brief}
\label{section-brief}

\noindent \textbf{Limit relations.} Many relations hold only in the limit. If each comparison carries a cost (the distance between the two sides of the $2$-morphism), one can ask that the infimum of the costs vanish; this is the language of categories enriched over $[0,\infty]$ (Section \ref{section-guide}). Applied to a transport datum with a metric on the observations, it gives equality of orbit closures \cite{MFK1994}, approximate unitary equivalence of operators (where the exact relation, unitary equivalence, is strictly finer in infinite dimension \cite{Voiculescu1976}), and the equivalence of Cauchy sequences by which the reals are built from the rationals.

\smallskip
\noindent \textbf{Stabilization relations.} As explained in Section \ref{section-guide}, two objects are \emph{stably} related if they become related after adding the same auxiliary summand. If the stabilization $\Sigma$ is compatible with the datum, the sequence $\mathfrak D\xrightarrow{\Sigma}\mathfrak D\xrightarrow{\Sigma}\cdots$ has a colimit which is again a comparison datum, and its relation $\simeq$ is the stable relation. This gives stable isomorphism of projective modules (equality in $K_0$), Witt equivalence of quadratic spaces \cite{MilnorHusemoller1973}, and stable homotopy type of finite complexes.

\smallskip
\noindent \textbf{Examples and counterexamples.} (i) In the Brandt monoid obtained by adjoining a unit to $\{(i,j)\}_{i,j\in\{1,2\}}\cup\{0\}$, the elements $(1,1)$ and $(1,2)$ are $\GJ$-related but multiplying on the right by $(2,2)$ gives $0$ and $(1,2)$, which are not; so $\GJ$ is not a congruence. (ii) For finite-dimensional vector spaces with the invertible Green datum ($\mathcal U_i$ the isomorphisms), $\simeq$ is equality of rank. (iii) On $\mathbb Z/3$ the relation $\{0,1\},\{2\}$ is completely functorial but neither a congruence nor a double coset relation, as shown in Section \ref{section-regularity}. (iv) Arbitrary $\sigma$ realizes the relation $E_0$ of Theorem \ref{theorem-nil}, which completely functorial data do not.

\section{Conclusion and open problems}
\label{section-conclusion}

\noindent Every comparison category is a Grothendieck construction restricted along a map $s$, and the classical relations are read off it: Green's relations are the orbit relations of the two-sided action of a category on its own morphisms, and orbit relations, hence all relations defined by an invariant, are the same construction after adjoining a point. The shape of the comparison category is unconstrained, so what constrains the relations that occur on a fixed category is the regularity of the transport maps $\sigma,\tau_1,\tau_2$. Rigid and diagonal data give Green-type relations and double cosets; completely functorial data contain Green's relations, orbit relations and congruences, depend on which morphisms are invertible: on finite groups, where all are, they are double cosets pulled back along homomorphisms; on nil monoids, where only the unit is, they collapse, which separates them from arbitrary $\sigma$.

\smallskip
\noindent \textbf{Open problems.} (1) Is every equivalence relation on a finite group the relation $\simeq$ of a completely functorial strict datum? (2) Is every equivalence relation on a monoid the connected-component relation (level $0$) of a completely functorial datum? (3) Do the pseudofunctorial and lax classes differ from the completely functorial one? (4) Is Green's $\GD$ the relation $\simeq$ of a single comparison datum of some regularity? (5) Does stabilization admit a description by a single transport datum, without passing to a colimit?

\begingroup\footnotesize
\section*{Declarations}
\vspace{-8pt}
\noindent\textbf{Generative AI use.} This article was generated by a generative artificial intelligence system (Claude Sonnet 5.5, Anthropic, October 2026) from an initial draft written by the author. The system restructured and rewrote the text, proposed and formulated the definitions, statements and proofs of the new results, carried out the supporting computations, and produced the \LaTeX{} source. The research questions, the initial draft and the scientific direction were provided by the author, who directed the generation through successive instructions and revisions. The author takes full responsibility for the content, and the AI system is not listed as an author.\\
\textbf{Ethics.} Not applicable. \textbf{Funding.} This research received no specific grant from any funding agency in the public, commercial, or not-for-profit sectors. \textbf{Competing interests.} The author declares no competing interests.

\endgroup


\begin{thebibliography}{99}
\setlength{\itemsep}{0pt}\setlength{\parsep}{0pt}\setlength{\topsep}{0pt}

\bibitem{Green1951} J.~A. Green, \emph{On the structure of semigroups}, Annals of Mathematics \textbf{54} (1951), no.~1, 163--172.

\bibitem{Howie1995} J.~M. Howie, \emph{Fundamentals of semigroup theory}, London Mathematical Society Monographs, Oxford University Press, 1995.

\bibitem{Lawvere1973} F.~W. Lawvere, \emph{Metric spaces, generalized logic, and closed categories}, Rend. Sem. Mat. Fis. Milano \textbf{43} (1973), 135--166.

\bibitem{MacLane1971} S.~Mac~Lane, \emph{Categories for the working mathematician}, Graduate Texts in Mathematics, vol.~5, Springer, 1971.

\bibitem{MFK1994} D.~Mumford, J.~Fogarty, F.~Kirwan, \emph{Geometric invariant theory}, 3rd ed., Springer, 1994.

\bibitem{MilnorHusemoller1973} J.~Milnor, D.~Husemoller, \emph{Symmetric bilinear forms}, Springer, 1973.

\bibitem{Thomason1979} R.~W. Thomason, \emph{Homotopy colimits in the category of small categories}, Math. Proc. Cambridge Philos. Soc. \textbf{85} (1979), 91--109.

\bibitem{Voiculescu1976} D.~Voiculescu, \emph{A non-commutative Weyl--von Neumann theorem}, Rev. Roumaine Math. Pures Appl. \textbf{21} (1976), 97--113.

\end{thebibliography}
\end{document}